\documentclass[11pt,twoside]{article}

\usepackage{mathrsfs,amsfonts,amsmath,amssymb}
\usepackage{latexsym}
\newtheorem{satz}{Theorem}
\newtheorem{proposition}[satz]{Proposition}
\newtheorem{theorem}[satz]{Theorem}
\newtheorem{lemma}[satz]{Lemma}

\newtheorem{corollary}[satz]{Corollary}
\newtheorem{remark}[satz]{Remark}
\newtheorem{example}[satz]{Example}

\def\no{\noindent}
\def\sbeq{\subseteq}
\def\T{\mathsf{T}}
\def\N{\mathbb {N}}
\def\Z{\mathbb {Z}}
\def\F{\mathbb {F}}
\def\E{\mathsf{E}}
\def\e{\varepsilon}
\def\l{\lambda}

\def\a{\alpha}

\def\h{\widehat}

\def\d{\delta}

\def\({\big (}
\def\){\big )}
\def\g{\gamma}

\def\ls{\leqslant}

\def\dim{{\rm dim}}
\def\le{\leqslant}
\def\ge{\geqslant}
\def\_phi{\varphi}
\def\eps{\varepsilon}

\def\Gr{{\mathbf G}}
\def\FF{\widehat}
\def\k{\kappa}
\def\ov{\overline}

\def\t{\tilde}
\def\Span{{\rm Span\,}}
\def\f{{\mathbb F}}

\def\c{\circ}

\begin{document}

\title{\bf Additive dimension and a theorem of Sanders}

\author{ By\\  \\{\sc Tomasz Schoen\footnote{ The author is supported by NCN 2012/07/B/ST1/03556.} ~ and ~
Ilya D.~Shkredov\footnote{
This work was supported by grant RFFI NN
11-01-00759, Russian Government project 11.G34.31.0053,
Federal Program "Scientific and scientific--pedagogical staff of innovative Russia" 2009--2013,
grant mol\underline{ }a\underline{ }ved 12--01--33080
and
grant Leading Scientific Schools N 2519.2012.1.}}
}
\date{}

\maketitle

\begin{abstract}
    We prove some new bounds for the size of the maximal dissociated subset
    of
    structured (having small sumset, large energy and so on) subsets $A$ of an abelian group.
\end{abstract}

\bigskip

\section{Introduction}

Let $\Gr$ be an abelian group.
A finite set $\Lambda \sbeq \Gr$
is called {\it dissociated} if any equality of the form
$$\sum_{\l \in \Lambda}\e_{\l}\l=0$$
for $\e_{\l}\in \{-1,0,1\}$ implies $\e_{\l}=0$ for all $\l.$
In many problems of additive combinatorics
(see e.g. \cite{BK_AP3,BK_struct,Bourgain_AP3,Bourgain_AP3_new,c1,sanders-log,Sanders_0.75,Sanders_1})
it is important to control
the maximal (by cardinality) dissociated subset of $A$,
which we  call the (additive) {\it dimension} of the set $A$.
The first general theorem on dimension of so--called large spectrum of a set was obtained in \cite{c1}.
For further results see
\cite{Bourgain_AP3_new,BourgainA+A,sanders-shkredov,sanders-log,sy}.
Let us recall two theorems proved in  \cite{sanders-log} and \cite{sy}, respectively, which we will apply later.

\begin{theorem}
    Let $A,B\subseteq \Gr$ be finite sets and
    suppose that $|A+B| \le K|A|$.
    Then $\dim (B) \ll K \log |A|$.
\label{t:Sanders}
\end{theorem}


\begin{theorem}\label{t:shkredov-yekhanin}
    Let $A,B \subseteq \Gr$ be finite sets and
    suppose that $\E (A,B) \ge |A| |B|^2 /K.$ Then there exist a set $B_1 \subseteq B$
    such that $\dim (B_1) \ll K \log |A|$
    and
    \begin{equation}\label{f:E(A,B_1)}
        \E (A,B_1) \ge 2^{-5} \E(A,B).
    \end{equation}
    In particular, $|B_1| \ge 2^{-3} K^{-1/2} |B|$.
    If $B=A$ then $\E(B_1) \ge 2^{-10}\E(A)$ and, consequently, $|B_1| \ge 2^{-4} K^{-1/3} |A|$.
\label{t:E(A,B)}
\end{theorem}

The aim of this paper is to obtain some further estimates on dimensions of sets.
First of all, we give a simple combinatorial proof of Theorem \ref{t:Sanders}
and refine the result in the case of
small doubling constant $K$, see Theorem \ref{t:dim}.
Furthermore, we generalize Theorem \ref{t:E(A,B)} for the case of another sorts of energies
and improve it by finding  subset
 having even smaller additive dimension
(see Theorems \ref{t:small-dim-subset}, \ref{t:sy_improved'}).
In section 6 we present an application of our results.
In the last section we reformulate some results from the papers \cite{BK_AP3}, \cite{BK_struct}
in terms of the additive dimensions and prove them for general abelian groups.

The Polynomial Freiman--Ruzsa Conjecture (PFRC) roughly states (see for details \cite{GT_PFRC=U_k}) that
every set $A$ with $|A+A| \le K|A|$ contains  highly structured subset of size $|A|/K^{O(1)}$.


If PFRC holds then for every set $A$ with $|A+A| \le K|A|$
there is a subset $B \subseteq A$, $|B| \gg K^{-C} |A|$ with $\dim (A) \ll K^{o(1)} \log |A|$, as $K\rightarrow \infty.$
Our results provide  bounds  of the form $\dim (A) \ll K^{c} \log |A|$,
where $c>0$ are some constants, which are much weaker than one could expect from PFRC.
However, our theorems  are still applicable
because in our results
size of the set $B$ is large and explicit, which is crucial, for example,
in problems concerning sets without solutions to a linear equation, see
e.g.
\cite{BK_AP3,sanders-log}.

I.D.S. is grateful of S.V. Konyagin for very fruitful discussions.

\section{Notation}

Let $\Gr$ be an abelian group.
If $\Gr$ is finite then denote by $N$ the cardinality of $\Gr$.
It is well--known~\cite{Rudin_book} that the dual group $\FF{\Gr}$ is isomorphic to $\Gr$ in the case.
Let $f$ be a function from $\Gr$ to $\mathbb{C}.$  We denote the Fourier transform of $f$ by~$\FF{f},$
\begin{equation}\label{F:Fourier}
  \FF{f}(\xi) =  \sum_{x \in \Gr} f(x) e( -\xi \cdot x) \,,
\end{equation}
where $e(x) = e^{2\pi i x}$.
We rely on the following basic identities
\begin{equation}\label{F_Par}
    \sum_{x\in \Gr} |f(x)|^2
        =
            \frac{1}{N} \sum_{\xi \in \FF{\Gr}} \big|\widehat{f} (\xi)\big|^2 \,.
\end{equation}
\begin{equation}\label{svertka}
    \sum_{y\in \Gr} \Big|\sum_{x\in \Gr} f(x) g(y-x) \Big|^2
        = \frac{1}{N} \sum_{\xi \in \FF{\Gr}} \big|\widehat{f} (\xi)\big|^2 \big|\widehat{g} (\xi)\big|^2 \,.
\end{equation}
and
\begin{equation}\label{f:inverse}
    f(x) = \frac{1}{N} \sum_{\xi \in \FF{\Gr}} \FF{f}(\xi) e(\xi \cdot x) \,.
\end{equation}
If
$$
    (f*g) (x) := \sum_{y\in \Gr} f(y) g(x-y) \quad \mbox{ and } \quad (f\circ g) (x) := \sum_{y\in \Gr} f(y) g(y+x)
$$
 then
\begin{equation}\label{f:F_svertka}
    \FF{f*g} = \FF{f} \FF{g} \quad \mbox{ and } \quad \FF{f \circ g} = \FF{f}^c \FF{g} = \ov{\FF{\ov{f}}} \FF{g} \,,
\end{equation}
where for a function $f:\Gr \to \mathbb{C}$ we put $f^c (x):= f(-x)$.
 Clearly,  $(f*g) (x) = (g*f) (x)$ and $(f\c g)(x) = (g \c f) (-x)$, $x\in \Gr$.
 The $k$--fold convolution, $k\in \N$  we denote by $*_k$,
 so $*_k := *(*_{k-1})$.

The characteristic function of a set
$S\subseteq \Gr$ we denote by $S(x).$
Write $\E(A,B)$ for the {\it additive energy} of sets $A,B \subseteq \Gr$
(see e.g. \cite{tv}), that is
$$
    \E(A,B) = |\{ a_1+b_1 = a_2+b_2 ~:~ a_1,a_2 \in A,\, b_1,b_2 \in B \}| \,.
$$
If $A=B$ we simply write $\E(A)$ instead of $\E(A,A).$
Clearly,
\begin{equation}\label{f:energy_convolution}
    \E(A,B) = \sum_x (A*B) (x)^2 = \sum_x (A \circ B) (x)^2 = \sum_x (A \circ A) (x) (B \circ B) (x)
    \,,
\end{equation}
and by (\ref{svertka}),
\begin{equation}\label{f:energy_Fourier}
    \E(A,B) = \frac{1}{N} \sum_{\xi} |\FF{A} (\xi)|^2 |\FF{B} (\xi)|^2 \,.
\end{equation}
Let
$$
   \T_k (A) := | \{ a_1 + \dots + a_k = a'_1 + \dots + a'_k  ~:~ a_1, \dots, a_k, a'_1,\dots,a'_k \in A \} | \,.
$$
Clearly, $\T_k (A) = \frac{1}{N}  \sum_{\xi} |\FF{A} (\xi)|^{2k}$.
For $A_1,\dots,A_{2k} \subseteq \Gr$ let
$$
    \T_k (A_1,\dots,A_{2k}) := | \{ a_1 + \dots + a_k = a_{k+1} + \dots + a_{2k}  ~:~ a_i \in A_i,\, i \in [2k] \} | \,.
$$
Put also
$$
    \sigma_k (A) := (A*_k A)(0)=| \{ a_1 + \dots + a_k = 0 ~:~ a_1, \dots, a_k \in A \} | \,.
$$
Notice that for a symmetric set $A$ that is $A=-A$ one has $\sigma_2(A) = |A|$ and $\sigma_{2k} (A) = \T_k (A)$.

 For a sequence $s=(s_1,\dots, s_{k-1})$ put
$A^B_s = B \cap (A-s_1)\dots \cap (A-s_{k-1}).$
If $B=A$ then write $A_s$ for $A^A_s$.
 Let
\begin{equation}\label{f:E_k_preliminalies}
    \E_k(A)=\sum_{x\in \Gr} (A\c A)(x)^k = \sum_{s_1,\dots,s_{k-1} \in \Gr} |A_s|^2
\end{equation}
and
\begin{equation}\label{f:E_k_preliminalies_B}
\E_k(A,B)=\sum_{x\in \Gr} (A\c A)(x) (B\c B)(x)^{k-1} = \sum_{s_1,\dots,s_{k-1} \in \Gr} |B^A_s|^2
\end{equation}
be the higher energies of $A$ and $B$.
The second formulas in (\ref{f:E_k_preliminalies}), (\ref{f:E_k_preliminalies_B})
can be considered as the definitions of $\E_k(A)$, $\E_k(A,B)$ for non integer $k$, $k\ge 1$.


\bigskip

For a positive integer $n,$ we set $[n]=\{1,\ldots,n\}$.
All logarithms used in the paper are to base $2.$
Signs  $\ll$ and $\gg$ are the usual Vinogradov symbols.
If $p$ is a prime number then write $\F_p$ for $\Z/p\Z$ and
$\F^*_p$ for $(\Z/p\Z) \setminus \{ 0 \}$.


\section{Preliminaries}

In this section we recall some  results, which we will need in the paper.


First of all, it was proved by Rudin \cite{Rudin_book} that all norms of Fourier transform
of a function with support on a dissociated set are equivalent.

\begin{lemma}\label{l:rudin}Let $\Lambda\sbeq \Z_N$ be a dissociated set and let $a_n$ be any
complex numbers. Then for each $p\ge 2$ we have
$$\frac1N\sum_{x=0}^{N-1}|\sum_{n\in \Lambda}a_n e^{-2\pi inx/N}|^p
\le (Cp)^{p/2}\Big(\sum_{n\in \Lambda}|a_n|^2\Big)^{p/2},$$ for some
absolute constant $C.$
\label{l:Rudin_L_p_L_2}
\end{lemma}

A consequence of the above lemma
is the following result due  of Sanders~\cite{sanders-shkredov}.
(Similar results
 were obtained by Bourgain~\cite{BourgainA+A} and by the second author~\cite{Sh_doubling}.)

\begin{lemma}
    Let $\Gr$ be a finite abelian group, $Q\subseteq \Gr$ be a set and $l$ be a positive integer.
    There is a partition $Q=Q^{str} \cup Q^{diss}$
    such that  $\dim (Q^{str}) < l$ and $Q^{diss}$ is a union of dissociated sets of sizes $l$.
    Moreover for all $p\ge 2$
    the following holds
    \begin{equation}\label{}
        \Big( \frac{1}{N} \sum_\xi | \FF{Q}^{diss} (\xi) |^p \Big)^{1/p}
            \ll
                \sqrt{p/l} \cdot |Q| \,.
    \end{equation}
\label{l:Bourgain_diss_sets}
\end{lemma}
We will also make use of a known  Chang's Covering Lemma \cite{c1} and
Pl\"{u}nnecke--Ruzsa inequality, see \cite{Ruzsa_book}, \cite{Ruzsa_card} or \cite{tv}.

\begin{lemma}\label{l:chang-covering}
    Let $L,K$ be real numbers, and $A,B \subseteq \Gr$ be two sets.
 If $|A+A|\le K|A|$ and $|A+B|\le L|B|$ then there are sets $S_1,\dots,S_l$ each of size at most $2K$ such
that $A\sbeq B-B+(S_1-S_1)+\dots +(S_l-S_l)$ and $l\le \log(2KL).$
\end{lemma}

\begin{lemma}\label{l:plunnecke-ruzsa}
Let $j<k$ be positive integers.
Let also $A,B$ be finite set of an abelian group such that $|A+jB|\le K|A|.$
Then there is a nonempty set $X\sbeq A$ such that
\begin{equation}\label{f:plunnecke-ruzsa1}
    |X+kB|\le K^{k/j}|X| \,.
\end{equation}
In particular, if $|A+A| \le K|A|$ then
\begin{equation}\label{f:plunnecke-ruzsa2}
    |mA-nA| \le K^{n+m} |A|
\end{equation}
for all $n,m\in \N$. Furthermore, for
fixed $j\ge 1$ and arbitrary $0<\d<1$ there exists   $X\subseteq A$ such that
$|X| \ge (1-\d) |A|$ and
\begin{equation}\label{f:plunnecke-ruzsa1'}
    |X+kB|\le (K/\d)^{k}  |X| \,.
\end{equation}
\end{lemma}

 We will also make use of   some results concerning  higher additive energies (see \cite{schoen-shkredov-higher} and \cite{sv}).

\begin{lemma}\label{l:E_k-identity} Let $A$ be a subset of an abelian group.
Then for every $k,l\in \N$
$$\sum_{s,t:\atop \|s\|=k-1,\, \|t\|=l-1} \E(A_s,A_t)=\E_{k+l}(A) \,,$$
where $\|x\|$ denote the number of components of vector $x$.
\end{lemma}


\begin{theorem}
    Let $A$ be a finite subset of an abelian group.
    Suppose that $\E(A)=|A|^3/K,$ and $\E_{3+\eps} = M|A|^{4+\eps}/K^{2+\eps}$, where $\eps\in (0,1]$.
    Put $P := \{ x : (A\circ A)(x) \ge |A|/2K \}$.
    Then $|P| \gg K|A| / M^{2/(1+\eps)}$ and  $\E(P) \gg M^{-\beta} |P|^3$, where $\beta = (3+4\eps) / (\eps(1+\eps))$.
\label{t:P_energy}
\end{theorem}

\begin{theorem}
    Let $A \subseteq \Gr$ be a finite set, and $l\ge 2$ be a positive integer.
    Then
    \begin{equation}\label{f:E_k_and_T_k_difference}
        \left( \frac{|A|^8}{8 \E_3(A)} \right)^l \le \T_l (A) |A-A|^{2l+1} \,,
    \end{equation}
\label{p:E_k_and_T_k}
\end{theorem}

\noindent Another relations between $\E_s$ and $\T_s$ can be found in \cite{schoen-shkredov-higher}, \cite{Sh_mixed}.

\begin{theorem}
    Let $A\subseteq \Gr$ be a finite set.
    Suppose that $\E_s (A) = |A|^{s+1}/K^{s-1}$, $s\in (1,3/2]$ and
 $\T_4 (A) := M|A|^7/K^3$, then
 $$
 \E_4 (A)
        \gg_{s-1} \frac{|A|^5}{M K }  \,.
    $$
\label{p:E_4_and_T_4}
\end{theorem}

By an arithmetic progression of dimension $d$ and size $L$ we mean a  set of the form
\begin{equation}\label{def:P_d}
    Q = \{ a_0 + a_1 x_1 + \dots + a_d x_d ~:~ 0\le x_j < l_j \} \,,
\end{equation}
where $L:=l_1 \dots l_d$.
$Q$ is said to be {\it proper} if all of the sums in (\ref{def:P_d}) are distinct.
By a {\it proper coset progression} of dimension $d$ we will mean a subset of $\Gr$
of the form $Q + H$, where $H \subseteq \Gr$ is a subgroup,
$Q$ is a proper progression of dimension $d$
and the sum is direct in the sense that $q+h = q' +h'$ if and only if $h = h'$ and $q = q'$.
By the size of a proper coset progression we mean simply its cardinality.

Finally, let us recall the main result proved in  \cite{sanders-bogolubov}.


\begin{theorem}
    Suppose that $\Gr$ is an abelian group and  $A,S\subseteq \Gr$ are finite non--empty sets such that
    $|A+S| \ls K \min \{ |A|, |S| \}$.
    Then $(A-A)+(S-S)$ contains a proper symmetric $d(K)$--dimensional coset progression $P$
    of size $\exp (-h(K)) |A+S|$.
    Moreover, we may take $d(K) = O(\log^6 K)$, and $h(K) = O(\log^6 K \log \log K)$.
\label{t:sanders-bogolubov}
\end{theorem}

\section{Additive dimension of sets with small doubling}

At the beginning
we derive a consequence of Theorem \ref{t:sanders-bogolubov}.

\begin{lemma}\label{l:new-green-ruzsa}
Let $A$ be a subset of an abelian group such that $|A+A|\le K|A|.$
Then
$$|kA|\le \Big(\frac{3ek}{K}\Big)^{O(K \log^{8} K)}|A|,$$
for every $k\ge K.$
\end{lemma}
\begin{proof}
By Theorem \ref{t:sanders-bogolubov} there exists a proper generalized arithmetic progression $P$
of dimension
$d\ll \log^7 K$ and size at least $|A|/K^{O(\log^7 K)}$ such that $P\sbeq 2A-2A.$
Thus, applying Pl\"{u}nnecke--Ruzsa inequality (\ref{f:plunnecke-ruzsa2}), we have
$$|A+P|\le |3A-2A|\le K^{5}|A|\le K^{O(\log^7 K)}|P| \,.$$
By Lemma \ref{l:chang-covering}, we obtain
$$A\sbeq P-P+(S_1-S_1)+\dots +(S_l-S_l)$$
with $l\ll (\log 2K)^{8}.$
Therefore,
\begin{eqnarray*}
|kA|&\le& |kP-kP+(kS_1-kS_1)+\dots +(kS_l-kS_l)|\le \binom{2K+k-1}{k}^{2l}|kP-kP|\\
&\le& \binom{2K+k-1}{k}^{2l}(2k)^d|P|\le \Big(\frac{3ek}{2K-1}\Big)^{4lK}(2k)^d|P|
    \le\Big(\frac{3ek}{K}\Big)^{O(K \log^{8} K)}|A| \,,
\end{eqnarray*}
provided that $k\ge K.\hfill\Box$
\end{proof}

\bigskip

\noindent Our first result refines Sanders' Theorem \ref{t:Sanders}, provided that $K$ is not too large.

\begin{theorem}\label{t:dim}
Let $A\subseteq \Gr$ be a finite set and
suppose that $|A+A|\le K|A|.$
Then, we have $$\dim(A)\ll K\log |A| \,$$
and
\begin{equation}\label{f:dim_1}
    \dim(A)\ll \log |A|+K(\log K)^{8} \log \log {|A|} \,.
\end{equation}
\end{theorem}
\begin{proof}
Let $\Lambda\sbeq A$ be a dissociated set such that $|\Lambda|=\dim(A).$  Then
by Lemma \ref{l:rudin} (or simple counting arguments)
and Lemma \ref{l:plunnecke-ruzsa} we have for some absolute constant $C>0$
$$\frac{|\Lambda|^{k}}{(C k)^k}\le |k\Lambda|\le |kA|\le K^k|A|.$$
Taking $k\sim \log |A|$ we obtain the first assertion.

Similarly, by Lemma \ref{l:new-green-ruzsa} we have
$$\frac{|\Lambda|^{k}}{(C k)^k}\le \Big(\frac{3ek}{K}\Big)^{O(K \log^{8} K)}|A| \,,$$
for  $k\ge K.$
Thus, putting
$$
    k\sim \log |A|+K(\log K)^{8} \log \log {|A|} \,,
$$
we get
$$|\Lambda|\ll \log |A| +K(\log K)^{8} \log \log {|A|}  $$
as required.
$\hfill\Box$
\end{proof}

\bigskip

In the above proof the hardest case is when the size of $kA$ attains its maximal value $K^k|A|.$
However, we show that if it is the case then one can find a huge subset of $A$ with very small additive dimension.

\begin{theorem}\label{t:dim-large sumset}
Let $A\subseteq \Gr$ be a set and $K\ge 1$, $\eps>0$ be real numbers.
Suppose that $|A+A|\le K|A|$ and $|kA|\ge K^{k-\e}|A|$ for some $k\ge 3.$
Then there exists a set $A'\sbeq A$ of
size at least $|A|/2$ such that $\dim(A') \ll 2^k K^\e \log |kA|.$
\end{theorem}
\begin{proof} From Lemma \ref{l:plunnecke-ruzsa}
it follows  that there exists a
set $X,\, |X|\ge |A|/2$ such that $|X+kA|\le (2K)^k |A|.$
Therefore
\begin{equation}\label{tmp:12.06.2013_1}
    |X+kA|\le 2^kK^{\e}|kA| \,.
\end{equation}
By Sanders' Theorem \ref{t:Sanders}, we have
$\dim(X)\ll  2^kK^{\e} \log|kA|.$
This completes the proof.$\hfill\Box$
\end{proof}



\bigskip

\noindent Recall the main result of \cite{schoen-freiman} (see Lemma 3).

\begin{theorem}\label{t:ts-PFRC} Let $A$ be a finite set of an abelian group such that $|A+A|\le K|A|.$ Then for every $k\in \N$
there exist sets $X\sbeq A$ and $Y\sbeq A+A$ such that  $|X|\ge (2K)^{-2^{k+1}}|A|,$ $|Y|\ge |A|$ and
$\E(X,Y)\ge K^{-2/k}|X|^2|Y|.$
\end{theorem}

\noindent Combining Theorem \ref{t:ts-PFRC} with Theorem \ref{t:E(A,B)}, we obtain the following consequence.

\begin{corollary}\label{c:small-dim-ts-PFRC}
Let $A$ be a finite set of an abelian group such that $|A+A|\le K|A|.$ Then for every $k\in \N$
there exists set $X\sbeq A$ such that $|X|\gg (2K)^{-2^{k+1}}|A| / K^{1/2}$ and $\dim(X)\ll K^{2/k}\log |A|.$
\end{corollary}

Using a well-known Croot-Sisask Lemma Sanders proved the following result (Proposition 3.1 in \cite{sanders-bogolubov}).

\begin{theorem}\label{t:sanders-X}
Let $A$ be a finite subset of an abelian group with $|A+A|\le K|A|$. Then for every $k\in \N$
there exists a set $X\sbeq A-t$ for some $t$, of size at least $e^{-O(k^2\log^22K)}|A+A|$ such that
$kX\sbeq 2A-2A.$
\end{theorem}

Applying  Theorem \ref{t:sanders-X} we show that every set with small sumset contains relatively large subset with very small additive dimension.

\begin{corollary}\label{c:small-dim-sanders}
Let $A$ be a finite subset of an abelian group with $|A+A|\le K|A|$. Then for every $k\in \N$
there exists a set $X\sbeq A$ of size at least $e^{-O(k^2\log^22K)}|A+A|$ such that
$$\dim(X)\ll K^{4/k}\log |A|.$$
\end{corollary}
\begin{proof} Observe that we can assume that $k\le \log |A|,$ because otherwise our theorem is trivial.  Let $X$ be the set given by Theorem \ref{t:sanders-X}. By Pl\"unnecke-Ruzsa inequality we have
$$|klX|\le |2lA-2lA|\le K^{4l}|A|.$$
Now, we argue as in Theorem \ref{t:dim}.
Let $\Lambda\sbeq X$ be a dissociated set with $|\Lambda|=\dim(X).$
By Rudin's inequality, we have for some absolute constant $C_1>0$
$$\frac{|\Lambda|^{kl}}{(C_1 kl)^{kl}}\le |kl\Lambda|\le |klX|\le  K^{4l}|A|.$$
Putting $l=[\frac{\log |A|}{k}]$ we see that
$$\dim(X)\ll K^{4/k}\log |A|,$$
which completes the proof.
$\hfill\Box$
\end{proof}

\section{Additive dimension of sets with large additive energy}

The aim of this section is to refine Theorem \ref{t:shkredov-yekhanin}, in the sense,
that under the same assumption $\E(A)=|A|^3/K$, we find a possibly large subset of $A$
having additive dimension  $O(K^{1-\g} \log |A|)$, where $\g>0$ is an absolute constant.

Our first result refines Theorem \ref{t:shkredov-yekhanin}.

\begin{theorem}\label{t:shkredov-yekhanin'}
    Let  $A,B$ be subsets of finite abelian group, and
   let $\eps >0$.
    Suppose that $\E (A,B) = |A| |B|^2 /K,$ then there exist a set $B_* \subseteq B$
    such that
    \begin{equation}\label{f:dim_sy_new}
        \dim (B_*) \ll_\eps K (\log K)^{2+\eps} \log |A| \cdot \left( \frac{|B_*|}{|B|} \right)^2 \,,
    \end{equation}
    and
    \begin{equation}\label{f:E(A,B_1)'}
        \E (A,B_*) \ge 2^{-2} \E(A,B) \,.
    \end{equation}
\label{t:E(A,B)'}
\end{theorem}
\begin{proof}
We establish  Theorem \ref{t:shkredov-yekhanin'} using the following algorithm.

At zero step we put $B_0 := B$, $\eps_0 (x) = 0$ and $\beta_0 = 1$.
At step $j\ge 1$ we apply Lemma \ref{l:Bourgain_diss_sets} to the set $B_{j-1}$ with
parameters $p=2+\log |A|$ and
$$l_j = \eta^{-1}K(\log K)\beta^2_{j-1} j^{1+\e} \log |A|\,,$$
where $\eta^{-1}=2^{4}\sum_{j\ge 1}1/j^{1+\e}.$
Lemma \ref{l:Bourgain_diss_sets} gives us a new set $B_j \subseteq B_{j-1}$,  where
$B_j = B^{str}$, in other words $B_{j-1} \setminus B_j$ is a disjoint union of all dissociated subsets each of size $l_j$.
After that put  $\eps_j (x) = B_{j-1}(x) - B_j(x)$, $\beta_j = |B_j| / |B|$ and iterate the procedure.
The described algorithm will satisfy  the following property
\begin{equation}\label{cond:algorithm}
    \E (A,B_j) \ge 2^{-2} \E(A,B) \,.
\end{equation}
Obviously, at the first step inequality (\ref{cond:algorithm}) is satisfied.
If at some step $j$ we get $\beta_{j} \ge \frac12 \beta_{j-1}$ then our
algorithm terminates with the output $B_* = B_{j}$.
In view of inequality (\ref{cond:algorithm}) it is clear the total number of steps $k$
does not exceed $\log K$.
Further, if our iteration procedure terminates with the output $B_{*}$,
then $\E (A,B_*) \ge 2^{-2} \E(A,B)$ and
\begin{eqnarray*}
    \dim (B_*)& =& \dim (B_j) \le l_j = \eta^{-1} K(\log K)\beta^2_{j-1} j^{1+\e} \log |A|\\
        &\le&
             2\eta^{-1}K(\log K)^{2+\e}\beta^2_{j} \log |A|\\
&\ll_\eps& K (\log K)^{2+\eps} \log |A| \cdot \left( \frac{|B_*|}{|B|} \right)^2 \,.
\end{eqnarray*}
Thus, the constructed set $B_*$ satisfies (\ref{f:dim_sy_new}), (\ref{f:E(A,B_1)'}).

It remains to check (\ref{cond:algorithm}), and clearly, it is sufficient to do it for the final step $k$.
We have
\begin{eqnarray}\label{f:E_basis_f}
N \cdot \E(A,B)& = &\sum_{\xi}|\FF{A} (\xi)|^{2} |\FF{B} (\xi)|^2
        =
            \sum_{\xi} |\FF{A} (\xi)|^{2} |\FF{B}_k (\xi)|^2+\nonumber \\
&+  &\Big(\sum_{j=1}^k \sum_{\xi} |\FF{A} (\xi)|^{2} \ov{\FF{B}_j (\xi)} \FF{\eps}_j (\xi)
                            +
                        \sum_{j=1}^k \sum_{\xi} |\FF{A} (\xi)|^{2} \FF{B}_j (\xi) \ov{\FF{\eps}_j (\xi)} \Big)  +
                        \sum_{j=1}^k \sum_{\xi} |\FF{A} (\xi)|^{2} |\FF{\eps}_j (\xi)|^2\nonumber \\
                       & =&         \sigma_0 + \sigma_1 + \sigma_2 \,.
\end{eqnarray}
By the H\"{o}lder inequality, the Parseval identity,
and our choice of parameters, we get
\begin{eqnarray*}
    \sigma_2
        &\le&
            \sum_{j=1}^k \Big( \sum_{\xi} |\FF{\eps}_j (\xi)|^{2p} \Big)^{1/p}
                \cdot
                    \Big( \sum_{\xi} |\FF{A} (\xi)|^{\frac{2p}{p-1}} \Big)^{1-1/p}\nonumber\\
                                &\le&
                                   \sum_{j=1}^k |A|^{1+1/p} |B_{j-1}|^2 N \frac{p}{l_j}
                                       \ll
                                          \eta (\log K)^{-1} K^{-1} |A| |B|^2 N \sum_{j=1}^k j^{-1-\e} \\
                                                &\le& 2^{-4} k^{-1} K^{-1} |A| |B|^2 N \,.
\end{eqnarray*}
Next, we  estimate $\sigma_1$ in a  similar way.
Let us consider only the first  term in $\sigma_1$, the second one can be bounded in the same manner.
By the Cauchy--Schwarz inequality, we get
\begin{eqnarray*}
    N^{-1} \left| \sum_{j=1}^k \sum_{\xi} |\FF{A} (\xi)|^{2} \ov{\FF{B}_j (\xi)} \FF{\eps}_j (\xi) \right|
       & \le&
            \sum_{j=1}^k \E^{1/2} (A,B_j) \E^{1/2} (A,\eps_j)\\
                &\le&
                    \Big( \sum_{j=1}^k \E(A,B_j) \Big)^{1/2}
                        \cdot
                    \Big( \sum_{j=1}^k \E(A,\eps_j) \Big)^{1/2}\\
&\le& k^{1/2} \E^{1/2} (A,B) \sigma^{1/2}_2
            \le
                2^{-2} \E (A,B) \,.
\end{eqnarray*}
So, by (\ref{f:E_basis_f}), we obtain $\E(A,B_k) \ge 2^{-2} K^{-1}|A||B|^2$.
This completes the proof.
$\hfill\Box$
\end{proof}

\begin{theorem}\label{t:small-dim-subset}
    Let $A$ be a finite subset of an abelian group.
Suppose that $\E(A)=|A|^3/K,$ then
there exists a set $B \sbeq A$ such that $|B|\gg |A|/K^{25/8}$
and
$
    \dim (B) \ll K^{7/8} \log |A|  \,.
$
\end{theorem}
\begin{proof}
Let  $\E_4(A)=M|A|^5/K^3$
and  $P := \{ x : (A\c A)(x) \ge |A|/2K \}$.
By Theorem \ref{t:P_energy},  $ |P|\gg K|A|/M$ and $\E(P)\gg |P|^3/M^{7/2}$.
By Theorem \ref{t:shkredov-yekhanin}
there exists $P'\sbeq P$ of size $\gg |P|/M^{7/6}$ such that $\dim(P')\ll M^{7/2}\log |P|.$
We have
$$\sum_{x\in A}(P'\c A)(x)\ge \frac{|A|}{2K}|P'|\gg \frac{|A|}{K}\frac{|P|}{M^{7/6}}\gg \frac{|A|^2}{M^{13/6}}\,. $$
Therefore,
$(P'\c A)(x) \gg |A|/M^{13/6}$ for some $x$.
Putting   $B=A\cap (P'+x)$ we see that
\begin{equation}\label{f:B_lower}
|B| \gg |A|/M^{13/6}
\end{equation}
and
\begin{equation}\label{tmp:12.06.2013_2}
    \dim(B)\ll M^{7/2}\log (K|A|) \ll M^{7/2}\log |A| \,.
\end{equation}

On the other hand, by Lemma \ref{l:E_k-identity}
\begin{eqnarray*}
\k_4 |A|^5&=& \frac{M|A|^5}{K^3} = \E_4(A)=\sum_{\|s\|=2} \E(A,A_{s})\le 2\sum_{|A_{s}|\ge \frac12\k_4|A|} \E(A,A_{s})\\
&\le& 2\max_{|A_{s}|\ge \frac12\k_4|A|}\frac{\E(A,A_{s})}{|A_{s}|^2} \cdot \sum |A_{s}|^2
    \le  \max_{|A_{s}|\ge \frac12\k_4|A|}\frac{\E(A,A_{s})}{|A_{s}|^2} \cdot \E_3(A)
\end{eqnarray*}
and similarly
$$\k_3 |A|^4=\E_3(A)=\sum_{\|t\|=1} \E(A,A_{t})\le 2  \max_{|A_{t}|\ge \frac12\k_3|A|}\frac{\E(A,A_{t})}{|A_{t}|^2} \cdot \E(A)\,.$$
so there exist $|A_{s}|\ge  \frac12\k_4|A|$ and $|A_t|\ge  \frac12\k_3|A|$ such that $\E(A,A_{s})\gg \k_4\k_3^{-1}|A||A_{s}|^2$ and $\E(A,A_{t})\gg \k_3K|A||A_t|^2.$
But $\k_4\k_3^{-1}\k_3K=M/K^2$, so either $\k_4\k_3^{-1}\ge M^{1/2}/K $ or $\k_3K\ge  M^{1/2}/K$. Hence by Theorem \ref{t:shkredov-yekhanin}
there is a set $B\sbeq A$ such that
\begin{equation}\label{f:B_lower-1}
    \frac{|B|}{|A|} \gg \min \{ \k_4 (\k_4 \k^{-1}_3)^{1/2}, \k_3 (\k_3 K)^{1/2} \} \ge \min\{M^{3/2}K^{-7/2},M^{3/2} K^{-5/2}\} = M^{3/2}K^{-7/2} \,,
\end{equation}
 and
\begin{equation}\label{tmp:12.06.2013_3}
    \dim(B)\ll \frac{K}{M^{1/2}}\log |A| \,.
\end{equation}
Combining (\ref{tmp:12.06.2013_2}), (\ref{tmp:12.06.2013_3})
and  (\ref{f:B_lower}), (\ref{f:B_lower-1}), we obtain the required result.
$\hfill\Box$
\end{proof}

\bigskip

Clearly, using Theorem \ref{t:shkredov-yekhanin'} instead of Theorem \ref{t:shkredov-yekhanin}
in the proof
one can estimate the dimension of the set $B$ in terms of the size of $B$.

\bigskip

To prove the next result  we need a generalization of  Theorem \ref{t:shkredov-yekhanin} for the energies $\T_k (A)$.

\begin{proposition}\label{p:T_k}
    Let $A\subseteq \Gr$ be a finite set, $k\ge 2$ be a positive integer and suppose that
    $\T_k (A) = c |A|^{2k-1}$.
    Then there is a set $A_* \subseteq A$ such that
    \begin{equation}\label{f:T_k_T_k(A*)}
        \T_k (A,\dots, A, A_*, A_*) \ge 2^{-5} \T_k (A)
    \end{equation}
    and
    \begin{equation}\label{f:T_k_dim}
        \dim (A_*) \ll \frac{\T_{k-1} (A) |A|^2}{\T_k (A)} \log (c^{-1}|A|)\le c^{-1/(k-1)} \log (c^{-1}|A|) \,.
    \end{equation}
    In particular
   $     |A_*| \gg c^{1/(2k-1)} |A| .$
\label{pred:T_k}
\end{proposition}
\begin{proof}
For any $l\le k$ let $\T_l (A) = c_l |A|^{2l-1}$, hence  $c_{k} = c$.
By Fourier transform, we have
$$\T_k (A) = \frac{1}{N} \sum_{\xi} |\FF{A} (\xi)|^{2k}.$$
We apply Lemma \ref{l:Bourgain_diss_sets} to the set $A$ with
parameters $p=2+\log (c^{-1} |A|)$ and $l = \eta^{-1} c^{-1}
c_{k-1} \log (c^{-1} |A|)$, where $\eta >0$ is an appropriate
constant to be specified   later. Write $\eps (x) = A(x) - A_*(x)$, where
$A_* = A^{str}$, in other words $A\setminus A_*$ is a disjoint union of all dissociated subsets each of size $l$.
We have
\begin{eqnarray*}
N \cdot \T_k (A) &= &\sum_{\xi}|\FF{A} (\xi)|^{2k-2} |\FF{A} (\xi)|^2
        =
            \sum_{\xi} |\FF{A} (\xi)|^{2k-2} |\FF{A}_* (\xi)|^2+\\
&&+   \,   \Big(\sum_{\xi} |\FF{A} (\xi)|^{2k-2} \ov{\FF{A}_* (\xi)} \FF{\eps} (\xi)
                            +
                        \sum_{\xi} |\FF{A} (\xi)|^{2k-2} \FF{A}_* (\xi) \ov{\FF{\eps} (\xi)} \Big)  +                        \sum_{\xi} |\FF{A} (\xi)|^{2k-2} |\FF{\eps} (\xi)|^2\\
                        &=&         \sigma_0 + \sigma_1 + \sigma_2 \,.
\end{eqnarray*}
By the H\"{o}lder inequality, the Parseval identity,
and our choice of parameters, we get
\begin{eqnarray}\label{tmp:18.02.2010}
    \sigma_2
        &\le&
            \Big( \sum_{\xi} |\FF{\eps} (\xi)|^{2p} \Big)^{1/p}
                \cdot
                    \Big( \sum_{\xi} |\FF{A} (\xi)|^{\frac{(2k-2)p}{p-1}} \Big)^{1-1/p}\nonumber\\
            &\ll&
                            \frac{p}{l} |A|^2 \cdot \T_{k-1} (A) \left(\frac{|A|^{2k-2}}{\T_{k-1} (A)}\right)^{1/p} N\\
                                &\le&
                                    2^{-1} c_k |A|^{2k-1} N \,.
\end{eqnarray}
To obtain the last inequality, we have used a simple bound $\T_{k-1} (A)
\ge c|A|^{2k-3}$. Hence either $\sigma_0$ or $\sigma_1$ is at least
$2^{-2} c_k |A|^{2k-1} N$. In the first case we are done. In the
second case an application of the Cauchy--Schwarz inequality yields
$$
    2^{-6} N^2 \T^2_{k} (A)
        \le
            N \cdot \T_{k} (A,\dots,A,A_*, A_*) \cdot \sigma_2 \,.
$$
Combining the inequality above with~(\ref{tmp:18.02.2010}), we get
$$
    \T_{k} (A,\dots,A,A_*,A_*) \ge 2^{-5} \T_k (A) \,.
$$
Using the last estimate  and
the H\"{o}lder inequality, we see that  $|A_*|\gg c^{1/(2k-1)} |A|$. Furthermore, we have
$\dim (A_*) \le l = \eta^{-1} c^{-1} c_{k-1} \log (c^{-1} |A|),$ which proves the first inequality in
(\ref{f:T_k_T_k(A*)}).
Applying the
H\"{o}lder inequality again, we see that $c_{k-1} \le c^{\frac{k-2}{k-1}},$ which gives the second inequality in
(\ref{f:T_k_T_k(A*)}).
This completes the proof of Proposition \ref{pred:T_k}.
$\hfill\Box$
\end{proof}
\bigskip

\begin{remark}
    One can also obtain an asymmetric version of the result above as well as
    a variant of Theorem \ref{t:E(A,B)'} for the energies $\T_k.$
\end{remark}

Let us also remark that the bound on the size of $A_*$ in Proposition \ref{p:T_k} is sharp up to a constant factor
(see example at the end of  section 2 from \cite{sy}).
Indeed, let $\Gr = \f_2^n$, and
$A= H\cup \Lambda$, where $H$ is a subspace, $|H| \sim  c^{1/(2k-1)} |A|$,
$\Lambda$ is a dissociated set (basis) and $c$ is an appropriate parameter.
Then $\T_k (A) \ge \T_k (H) = |H|^{2k-1} \gg c |A|^{2k-1}$,
any set $A_* \subseteq A$ satisfying (\ref{f:T_k_dim}) has large intersection with $H$, hence it cannot have size much greater then $c^{1/(2k-1)} |A|$.

\bigskip

If one replace the condition of Theorem \ref{t:small-dim-subset} on  $\E(A)$
by a similar one  on $\E_{3/2} (A)$,  then
the the following result can be proved.

\begin{theorem}
 Let $A$ be a finite subset of an abelian group and
suppose that $\E_{3/2} (A)=|A|^{5/2}/K^{1/2}.$
Then
there exists a set $B \sbeq A$ such that
\begin{equation}\label{f:B_size_1.5}
    |B|\gg |A|/K^2
\end{equation}
 and
\begin{equation}\label{f:B_dim_1.5}
    \dim (B) \ll K^{3/4} \log |A| \,.
\end{equation}
\label{t:sy_improved'}
\end{theorem}
\begin{proof}
Write $\T_4 (A) = M |A|^7 /K^3$,  $M\ge 1,$ then
by Theorem \ref{p:E_4_and_T_4}, we have
$$
    \E_4 (A) := \k_4 |A|^5  \gg \frac{|A|^5}{MK} \,,
$$
Furthermore,
$$\sum_{|A_s|\le\frac14\k_4 |A|}\E(A_s, A_t)\le  \sum_{|A_s|\le\frac14\k_4 |A|} |A_s|^2|A_t|\le \frac14\E_4(A)$$
hence by Lemma \ref{l:E_k-identity}
$$\frac12\E_4(A)\le \sum_{|A_s|,\, |A_t|\ge \frac14\k_4 |A|}\E(A_s, A_t)\le \max_{|A_s|,\, |A_t|\ge
\frac14\k_4 |A|}\frac{\E(A_s, A_t)}{|A_s|^{3/2}|A_t|^{3/2}} \cdot \sum_{s,t}|A_s|^{3/2}|A_t|^{3/2} \,.$$
Therefore, there are
$|A_s|,|A_t| \gg \frac14\k_4 |A|$
such that
$$
    \E(A_s,A_t) \gg |A_s|^{3/2} |A_t|^{3/2} \cdot \frac{\E_4(A)}{\E_{3/2}(A)^2} \ge \frac{|A_s|^{3/2} |A_t|^{3/2}}{M }  \,
$$
and by the Cauchy--Schwarz inequality we see that either
$\E(A_s) \gg   |A_s|^{3}/M,$ or
$\E(A_t) \gg   |A_t|^{3}/M$.
Applying Theorem \ref{t:E(A,B)} in the symmetric case we find
$B \subseteq A$ such that
\begin{equation}\label{tmp:05.06.2013_1}
    |B|
    \gg \frac{\k_4 |A|}{M^{1/3}}
        \gg
            \frac{|A|}{M^{4/3} K} \,,
\end{equation}
and
\begin{equation}\label{tmp:05.06.2013_2}
    \dim (B) \ll M  \log |A| \,.
\end{equation}

On the other hand, using Proposition \ref{pred:T_k}, we get a set $B'\sbeq A$ such that
$|B'| \gg M^{1/7} K^{-3/7} |A|$
and
$$\dim (B') \ll K M^{-1/3} \log |A|.$$
Combining the last inequalities with (\ref{tmp:05.06.2013_1}), (\ref{tmp:05.06.2013_2}),
we obtain the required result.
$\hfill\Box$
\end{proof}

\bigskip

Again, using Theorem \ref{t:shkredov-yekhanin'} instead of Theorem \ref{t:shkredov-yekhanin}
in the proof
one can estimate the dimension of the set $B$ in terms of the size of $B$.

\bigskip


The last result of this section shows that  small $\E_3(A)$ energy
implies that a large subset of $A$ has small dimension.

\begin{theorem}
    Let $A$ be a finite subset of an abelian group.
Suppose that $|A-A|\le K|A|$ and $\E_{3}(A)=M^{}|A|^{4}/K^{2}.$ Then
there exists $A_* \sbeq A$ such that $|A_*|\gg |A|/M^{1/2}$
 and
$$
    \dim (A_*) \ll M (\log |A| +  \log K \log M) \,.
$$
\end{theorem}
\begin{proof}
By   Theorem
\ref{p:E_k_and_T_k} for every $l\ge 2$  we have $\T_l (A) \ge
|A|^{2l-1}/(K (8M)^l)$. Applying Proposition \ref{pred:T_k}  with $l\sim \log K$ we obtain the result. $\hfill\Box$
\end{proof}

\section{An application}

Konyagin posed the following interesting problem. Is it true that there is a constant $c>0$ such that if $A\sbeq \F_p$ and $|A|\le \sqrt p$ then there exists
$x$ such that $0<(A*A)(x)\ll |A|^{1-c}.$ First nontrivial  results toward this conjecture were obtained in
\cite{LS}.
It was proved that there exists
$x$ such that $0<(A*A)(x)\ll e^{-O((\log \log |A|)^2)}|A|^{},$ provided that $|A|\le e^{c\log ^{1/5}p}.$ Our next result
improves the above estimate
as well as the condition on size of $A$.

\begin{theorem}\label{t:Konyagin-problem} Suppose that $A\sbeq \F_p$ and $|A|\le e^{c\sqrt {\log p}}$. Then there exists $x$ such that
$$0<(A*A)(x)\ll e^{-O(\log^{1/4} |A| )}|A|$$
for some absolute constant $c>0.$
\end{theorem}
\begin{proof} Let us write $|A|/K=\min_{x\in A+A}(A*A)(x),$ then clearly
 $|A+A| \le K|A|$.
Let $X \subseteq A$ be a set given by Corollary \ref{c:small-dim-sanders}
for
$k=[\log K].$ Then $|X|\gg e^{-O({\log^4K})}|A|$
and $\dim(X) \ll \log |A|.$
 Suppose that $\Lambda$ satisfies $ |\Lambda|=\dim (X)$ and $X \sbeq \Span (\Lambda).$
By Dirichlet approximation theorem there exists $r\in \F_p^*$ such that
$$\|rt/p\|\le p^{-1/|\Lambda|} .$$ for every
$t\in \Lambda$ and therefore
\begin{eqnarray*}
\|rx/p\|&\le& |\Lambda| p^{-1/|\Lambda|}\ll (\log |A|)  p^{-O(1/\log |A|)}< \frac1{K|A|}
\end{eqnarray*}
for every $x\in X.$ We can assume that there is  a set $X'\sbeq X \sbeq A$ of size at least $|X|/2$ such that
for each $x\in X'$ we have $\{rx/p\}<1/(K|A|).$

 Notice  that for every $r\in \F_p^*$ there is a large gap in the set $r\cdot (A+A)$ i.e. there exists $s\in A+A$ such that
$$\{rs+1,\dots, rs+l\}\cap r\cdot(A+A)=\emptyset,$$
where $l=\frac{p-|A+A|}{|A+A|}\gg \frac{p}{K|A|}.$
Since $(A*A)(s)\ge |A|/K$ it follows that there are at least $|A|/K$ elements $a\in A$
such that
$$\{ra+1,\dots, ra+l\}\cap (r\cdot A)=\emptyset.$$ Denote the set of such $a$'s by $Y \subseteq A$.
Thus,
$$K|A|\ge |A+A|=|r\cdot A+r\cdot A|\ge |X'+Y|= |X'||Y|\ge e^{-O(\log^42K)}|A|^2,$$
so that $K\gg e^{O(\log^{1/4} |A| )}$, and the assertion follows.
$\hfill\Box$
\end{proof}

\bigskip

Bukh proved \cite{Bukh_dilates} that if $A\sbeq \Gr$ and $\l_i\in \Z\setminus \{0\}$ then
$$|\l_1\cdot A+\dots +\l_k\cdot A|\le K^{O(\sum_i\log(1+|\l_i|))}|A| \,,$$
where $K = |A\pm A|/|A|$.
We also prove here an estimate for sums of dilates.
It is not directly related with the additive dimension of sets but it is  another consequence of Theorem \ref{t:sanders-bogolubov}.

\begin{theorem}\label{t:Konyagin-problem} Let $A\sbeq \Gr$ is a finite set and $\l_i\in \Z\setminus \{0\}$. Suppose that $|A+A|\le K|A|$
then
$$|\l_1\cdot A+\dots +\l_k\cdot A|\le  e^{O((\log ^{8} K) (k+\log(\sum_i|\l_i|))}|A|.$$
\end{theorem}
\begin{proof} From Sanders' Theorem \ref{t:sanders-bogolubov} it follows that  there is a
 $O(\log^7K)-$dimensional arithmetic progression $P$ of size $|P|\gg |A|/K^{O(\log^7K)}$ contained in $2A-2A.$
 By the well-known Ruzsa covering  lemma
 there is a set $S$ with $|S|\ll K^{O(\log^7K)}$ such that
 $$A\sbeq S+P-P.$$
Therefore,
\begin{eqnarray*}
|\l_1\cdot A+\dots +\l_k\cdot A|&\le & K^{O(k\log ^{7} K)}|\l_1\cdot (P-P)+\dots +\l_k\cdot (P-P)|\\
&\le & K^{O(k\log ^{7} K)}|(\sum_{\l_i>0}\l_i)(P-P)+(\sum_{\l_i<0}\l_i)(P-P)|\\
&\le & K^{O(k\log ^{7} K)} (|\l_1|+\dots +|\l_k|)^{\log^7K}|P-P|\\
&\le &  e^{O((\log ^{8}K)(k+ \log(\sum_i|\l_i|))}|A| \,,
\end{eqnarray*}
which completes the proof.
$\hfill\Box$
\end{proof}

\section{A result of Bateman and Katz}

In this section we reformulate some results from \cite{BK_AP3,BK_struct} in terms of additive dimension.
Although in  \cite{BK_AP3,BK_struct} the authors have deal with the case $\Gr=\f_p^n$,
where $p$ is a prime number, it is easy to see that their arguments work in more general groups.
We will follow their argument with some modifications.

Let $A\subseteq \Gr$  and $s$ be a positive integer.
A $2s$--tuple $(x_1,\dots,x_{2s}) \in A^{2s}$ is called the {\it additive $2s$--tuple}
if $x_1+\dots+x_{s} = x_{s+1}+\dots+x_{2s}$.
We say that an additive $2s$--tuple
$(x_1,\dots,x_{2s})$
is
{\it trivial}
if at least two variables are equal.
Otherwise we say that $2s$--tuples is {\it nontrivial}.
 Let $\T^*_s (A)$ denotes the number of nontrivial $2s$--tuples.
We will often use the following inequality   $\T_l (A)^{s-1} \le \T_{s} (A)^{l-1} |A|^{s-l}$, which holds for every
$s\ge l \ge 2$.

\begin{lemma}\label{l:T^*_lower_bound}
    Let $A \subseteq \Gr$ and $s\ge 4$.
    Suppose that $\T_s (A) \gg 10^s s^{2s} |A|^s$.
    Then $\T^*_s (A) \ge \frac12\T_s (A) $.
\end{lemma}
\begin{proof} We proceed similarly like in the proof of Theorem 5.1 in \cite{Ruzsa-solving-I}. Let $(A\tilde *_sA)(x)$ denotes the number of representations $x=x_1+\dots+x_s$ in distinct $x_i\in A.$
 Observe that
$\sum _x(A\tilde *_sA)(x)^2$
equals $\T^*_s(A)$ plus the number of additive tuples $(x_1,\dots, x_{2s})$ such that for some $i\le s$ and $j>s$ we have $x_i=x_j.$ Hence,
\begin{equation}\label{f:distinct-sums-est}
\sum _x(A\tilde *_sA)(x)^2-\T^*_s(A)\le s^2|A|\sum _x(A\tilde *_{s-1}A)(x)^2\le s^2|A|\T_{s-1}(A).
\end{equation}
Notice that $(A *_sA)(x)-(A\tilde *_sA)(x)$ is the number of representations $x=x_1+\dots+x_s$, for which
$x_i=x_j$ for some $i<j.$ Thus, we have
\begin{equation}\label{f:distinct-sums-est-}
    (A *_sA)(x)-(A\tilde *_sA)(x)\le s^2q(x) \,,
\end{equation}
where  $q(x)$ is the number of solutions of $x=2x_1+\dots+x_{s-1}.$
By Fourier inversion
\begin{eqnarray}\label{f:q-est}
\sum_xq(x)^2&=& \int |\h A(2\a)|^2|\h A(\a)|^{2s-4} \,d\a
\le
|A|^2\T_{s-2}(A)\le |A|^{2+\frac{2}{s-1}}\T_{s}(A)^{\frac{s-3}{s-1}}\nonumber\\
&=&|A|^{2+\frac{2}{s-1}}\T_{s}(A)^{-\frac{2}{s-1}} \T_s(A)\le \frac1{100}s^{-4} \T_s(A) \,.
\end{eqnarray}
Therefore, by the triangle inequality
and inequalities (\ref{f:distinct-sums-est-}), (\ref{f:q-est}), we get
\begin{eqnarray}\label{tmp:10.11.2013_1}
\sum _x(A\tilde *_sA)(x)^2&\ge& \T_s(A)^{1/2}\Big (\T_s(A)^{1/2}-2s^2(\sum_x q(x)^2)^{1/2}\Big )\nonumber\\
&\ge& \T_s(A)^{1/2}\Big (\T_s(A)^{1/2}-\frac15\T_{s}(A)^{1/2}\Big)\nonumber\\
&\ge& \frac45\T_s(A).
\end{eqnarray}
Finally, using the assumption that $\T_s (A) \gg 10^s s^{2s} |A|^s$,
and bounds (\ref{f:distinct-sums-est}), (\ref{tmp:10.11.2013_1}), we obtain
$$\T^*_s(A)\ge \frac45\T_s(A)-s^2|A|\T_{s-1}(A)\ge \frac45\T_s(A)-s^2|A|^{\frac{s}{s-1}}\T_{s}(A)^{\frac{s-2}{s-1}}\ge
\frac12\T_s(A)$$
and the assertion follows.
$\hfill\Box$
\end{proof}
\bigskip

We will also use the following  simple  lemmas.

\begin{lemma}\label{l:tuples_l}
    Let  $A \subseteq \Gr$ be a finite set and let $s>0$ be an even integer.
    Suppose that $A$ contains
    a family of
    nontrivial $s$--tuples, involving
    at least
    $r s$ elements of $A$.
    Then $\dim (A) \le |A| - r$.
\end{lemma}
\begin{proof}
Let $\mathcal{S}$ denotes the given family of $s$--tuples and
let $M \subseteq A$ be the set of all elements of $A$ involved
in some $s$--tuple of $\mathcal{S}$.
To proof the lemma
it is sufficient to show that there are $s$--tuples $S_1,\dots,S_r\in \mathcal{S}$
and elements $a_j \in S_j$, $j\in [r]$ such that each $a_j$ does not belong to  $S_i$, $i\neq j$.
Indeed, it is easy to see that $A\sbeq \Span (A\setminus\{a_1,\dots,a_r\}).$

We use induction on $r\ge 0$. The result is trivial for $r=0.$ Now assume that $r\ge 1.$
In view of the assumption $|M| \ge rs$ there is an element $a\in M$ belonging at most
$k:= s|\mathcal{S}|/|M| \le  |\mathcal{S}|/r$
tuples from $\mathcal{S}$.
Let $S_1, \dots, S_k$ be all  these tuples and put $\mathcal{S'}=\mathcal{S}\setminus\{S_1, \dots, S_k\}.$
One can suppose that the minimum of such $k$ is attained on the element $a\in M$.
Notice that $\mathcal{S}'$ involves at least $rs - s$ elements of $A$.
Indeed, otherwise  $|S_1\cup\dots \cup S_k|\ge s+1$ and each element of $S_1\cup\dots \cup S_k$ belongs to at least
$k$ sets from $\mathcal{S}$, so that it belongs to all sets $S_1,\dots,S_k.$
But this implies that $|S_1\cup\dots \cup S_k| \le ks/k =s$, which gives a  contradiction.
By induction assumption there are tuples  $S'_1,\dots,S'_r\in \mathcal{S'}$
and elements $a'_j \in S'_j$, $j\le r-1$ such that each $a'_j$ does not belong to  $S'_i$, $i\neq j$.
Hence the sets $S_1,S'_1,\dots,S'_r\in \mathcal{S}$ and the elements $a_1, a_1',\dots,a'_{r-1}$ posses the required property.
$\hfill\Box$
\end{proof}

\begin{lemma}
    Let $M \subseteq \Gr$ be a finite set and suppose that $M=X\cup D$, where $D$ is a dissociated set. Then there is an absolute constant $C>0$ such that
    $\T_{s}(M)\le C^ss^s|D|^s+2^{2s}|X|^{2s-1}.$
   \label{l:large-diss-energy}
\end{lemma}
\begin{proof} By Rudin's inequality we have
$$\T_s(M)=\int |\h M(\a)|^{2s} \,d\a \le 2^{2s}\int |\h D(\a)|^{2s} \,d\a +2^{2s} \int |\h X(\a)|^{2s} \,d\a
    \le C^ss^s|D|^s+2^{2s}|X|^{2s-1}.
~~\hfill\Box$$
\end{proof}


\begin{proposition}
    Let $A\subseteq \Gr$ be a finite set such that $\T_k(A)\ge 10^k s^{2k}|A|^k,$ where $2\le k<s=\lfloor \log |A|\rfloor$. Furthermore, let $\sigma \ge 1$
and      $d$ be
  such that
    \begin{equation}\label{f:d_bound1}
         \frac{|A|^{1-\frac{s-k}{2s(k-1)}}\log^{3/2} |A|}{\T_k(A)^{\frac{s-1}{2s(k-1)}}}           \ll d
                \le \frac{|A|^{1/2}}{\sigma^{1/2}}.
    \end{equation}
    Then there is a
    set $A'\subseteq A$ such that
    $\dim (A') \le d$ and
    \begin{equation}\label{f:A_L_bound}
        |A'| \ge \sigma \dim (A') \,.
    \end{equation}
\label{p:finding_concentration}
\end{proposition}
\begin{proof}
Suppose that for all
sets
$A'\sbeq A$ such that  $\dim A' = m \le d$ we have $|A'| < m \sigma$.
We choose $d$ elements from $A$ uniformly and random.
We show that
\begin{equation}\label{f:nullity_l}
    \mathbb{P} \big ( \,\dim(\{x_1,\dots, x_d\})\le d-l\,\big)
        = O(l)^{-l}  \,.
\end{equation}
Indeed, suppose that we have chosen $x_1,\dots,x_m$ for some $m \le d.$
Put
$$
    A' := \Span(W) \cap A\,,
$$
where $W$ is a maximal dissociated subset of $\{x_1,\dots, x_m\}$.
Clearly, $|W| \le m$ and hence
$\dim(A')\le m$.
By our assumption $d \le \frac{|A|^{1/2}}{\sigma^{1/2}}$
and therefore the probability that $x_{m+1}$
belongs to $A'$ is at most $|A'|/|A|\le m\sigma /|A| \le d\sigma /|A| \le 1/d$. Observe that if $\dim(\{x_1,\dots, x_d\})\le d-l$
then there are at least $l$ elements $x_{i+1}$ such that
$x_{i+1}\in \Span(W_i) \cap A$, where $W_i$ is a maximal dissociated subset of $\{x_1,\dots, x_i\}$.
Thus, the required probability is bounded from above by
$$
    \sum_{j=l}^{d} \binom{d}{j} \frac{1}{d^j} \le \sum_{j=l}^{\infty} \left( \frac{ed}{j} \right)^j \frac{1}{d^j}
         = O(l)^{-l}
$$
and (\ref{f:nullity_l}) is proved.

Next, suppose that the  tuple $(x_1,\dots,x_d) \in A^d$ has dimension $d-l$.
Let  $M$ be the set that consists  of all elements of $\{x_1,\dots,x_d\}$,
which are involved in some nontrivial $2s-$tuple.
Then by Lemma \ref{l:tuples_l},
$|M|\le 2sl.$
Since
$M$
contains $|M|-l$ element dissociated subset it follows
by Lemma \ref{l:large-diss-energy} that $\T^*_{s}(M)\le \T_{s}(M)\le C^ss^{s}(2sl)^s+2^{2s}l^{2s-1}.$
Therefore, the expected number of nontrivial $2s-$tuples in $(x_1,\dots,x_d)$ is bounded from above by
\begin{equation}\label{f:number_of_2s-tuples}
C_1^s\sum_{l=0}^d(s^{2s}l^s+l^{2s-1})O(l)^{-l}\le C_2^s s^{3s} \,,
\end{equation}
where $C_1, C_2 >0$ are absolute constants.

On the other hand the expected  number of nontrivial
$2s-$tuples in $(x_1,\dots,x_d)$
equals $ \T^*_s (A)(d/|A|)^{2s}$ and by Lemma \ref{l:T^*_lower_bound} we have
$$
    \T^*_s (A) \left( \frac{d}{|A|} \right)^{2s}
        \ge
            \frac12 \T_s (A) \left( \frac{d}{|A|} \right)^{2s}
                \ge
                    \frac{\T_k(A)^{\frac{s-1}{k-1}}}{2|A|^{\frac{s-k}{k-1}}} \left( \frac{d}{|A|} \right)^{2s} \,.
$$
Comparing the last estimate with (\ref{f:number_of_2s-tuples}) (recalling that $s=\lfloor \log |A|\rfloor$),
we obtain a contradiction.
This completes the proof.
$\hfill\Box$
\end{proof}

\bigskip

Finally let us formulate the Bateman--Katz thorem for general abelian group $\Gr$.

\begin{corollary}\label{t:better_T_k}
    Let $A\subseteq \Gr$ be a finite set and let $k$ be a fixed integer, $2\le k<\lfloor \log |A|\rfloor$.
    Suppose that $\T_k (A) = c |A|^{2k-1} \ge 10^k  |A|^k \log^{2k} |A|$.
    Then there is a set $A' \subseteq A$ such that
    \begin{equation}\label{f:better_T_k_size}
        |A'| \gg
                \frac{c^{\frac{1}{k-1}} |A|}{\log^3 |A|}
            \cdot
                \log (c^{\frac{1}{k-1}}  |A|) \,,
    \end{equation}
    and
    \begin{equation}\label{f:better_T_k_dim}
        \dim (A') \ll_k c^{-\frac{1}{2(k-1)}} \cdot \log^{3/2} |A|
            \,.
    \end{equation}
\end{corollary}
\begin{proof}
As in Proposition \ref{p:finding_concentration} put $s=\lfloor \log |A|\rfloor$.
In view of  $k< s$, we have $\T_s (A) \gg 10^{s} s^{2s} |A|^s$ and
$\T_s (A) \ge c^{\frac{s-1}{k-1}}|A|^{2s-1}$.
We apply Proposition \ref{p:finding_concentration} with
$$
    d \sim |A| \T^{-1/2s}_s (A) \log^{3/2} |A| \ll c^{-1/2(k-1)} \log^{3/2} |A|
        ~~ \mbox{   and   } ~~
    \sigma \sim |A| d^{-2} \,.
$$
Then the conditions (\ref{f:d_bound1}) are satisfied.
Thus, there exists  a set $A' \subseteq A$ of dimension at most $d$
such that
$$
    |A'|
            \ge  \sigma\dim(A')
        \gg \sigma \log \sigma \gg
        \frac{c^{\frac{1}{k-1}} |A|}{\log^3 |A|}
            \cdot
                \log (c^{\frac{1}{k-1}}  |A|) \,.
$$
This completes the proof.
$\hfill\Box$
\end{proof}



\section{Further remarks}

\bigskip

We finish the paper with some remarks on other possible variants of additive dimension, which we considered here. For any $A\subseteq \Gr$ put
$$
    d(A) = \min \{ |S| ~:~ S\subseteq A,\, A \subseteq \Span S \}\,,
        \quad \quad
            d_*(A) = \min \{ |S| ~:~ A \subseteq \Span S \} \,,
$$
where for $S=\{ s_1,\dots, s_l \}$ we define
$$
    \Span (S) := \Big\{ \sum_{j=1}^{l} \eps_j s_j ~:~ \eps_j \in \{ 0,-1,1 \} \Big\} \,.
$$

\begin{example}
    Let $x\not =y$ be integers and
    let $A_1 = \{ x,y,x+y,2x+y \}$, $A_2 = \{ y,x+y,2x+y \}$.
    Clearly, $A_2 \subsetneq A_1$ and
     $\dim (A_1) = 3$, $d(A) = d_* (A_1) = 2$,
    $\dim (A_2) = d (A_2) = 3$, $d_* (A_2) = 2$,
    Thus every kind of dimension can  differ from another one.
    Note also that $d(A_2) > d(A_1)$,  but $A_2 \subsetneq A_1$.
\label{ex:dimensions}
\end{example}

Observe that
$$
    d_* (A) \le d(A) \le \t{d} (A) \le \dim(A) \,.
$$
On the other hand
$$
    \dim(A) \ll d_* (A) \log d_* (A) \le d (A) \log d (A)
$$
(see  \cite{LY_diss}).
Indeed,  let $\Lambda \subseteq A$ be a maximal dissociated subset of $A$, $|\Lambda| = \dim (A)$ and let $|S| = d_*(A)$.
There are $2^{|\Lambda|}$ different subset sums of $\Lambda$
and any element of $A$ and hence any element of $\Lambda$ belongs to $\Span S$, so that
\begin{equation*}\label{f:Lev}
    2^{|\Lambda|} \le (2|\Lambda| + 1)^{|S|}
\end{equation*}
and the result follows.

Each of the dimensions has useful properties: $\dim (A)$, $d_* (A)$ are monotone (but $d (A)$ is not as Example \ref{ex:dimensions} shows), furthermore  all dimensions are subadditive
$$
    \dim ( C_1 \cup \dots \cup C_n) \le \sum_{j=1}^n \dim (C_j)
$$
and the same holds for $d_* (A)$, $d (A)$
and
the dimension $d(A)$ is "subadditive" in the following sense
$$
    d(C_1+ \dots +C_n) \le \sum_{j=1}^n d (C_j)
$$
for any disjoint sets $C_j$.
There are another dimensions, e.g.
$$
    \t{d} (A) = \min \{ |\Lambda| ~:~ \Lambda \subseteq A \,, \quad \Lambda \mbox{ is maximal (by inclusion) dissociated subset of } A \} \,,
$$
Clearly, $d(A) \le \t{d} (A) \le \dim (A)$ and $\t{d}$ is monotone because of $3=\t{d}(A_2) > \t{d}(A_1)=2$.

{}

\bigskip

\no{Faculty of Mathematics and Computer Science,\\ Adam Mickiewicz
University,\\ Umul\-towska 87, 61-614 Pozna\'n, Poland\\} {\tt
schoen@amu.edu.pl}

\bigskip

\no{Division of Algebra and Number Theory,\\ Steklov Mathematical
Institute,\\
ul. Gubkina, 8, Moscow, Russia, 119991\\}
and
\\
Delone Laboratory of Discrete and Computational Geometry,\\
Yaroslavl State University,\\
Sovetskaya str. 14, Yaroslavl, Russia, 150000\\
and
\\
IITP RAS,  \\
Bolshoy Karetny per. 19, Moscow, Russia, 127994\\
{\tt ilya.shkredov@gmail.com}


\begin{thebibliography}{99}


\bibitem{BK_AP3}
{\sc M.~Bateman, N.~Katz, }
{\em New bounds on cap sets, }
Journal of AMS {\bf 25}:2 (2012), 585--613.


\bibitem{BK_struct}
{\sc M.~Bateman, N.~Katz, }
{\em Structure in additively nonsmoothing sets, }
arXiv:1104.2862v1 [math.CO] 14 Apr 2011.


\bibitem{Bourgain_AP3}
{\sc J.~Bourgain, }
{\em On triples in arithmetic progression, }
Geom. Funct. Anal. {\bf 9} (1999), 968--984.


\bibitem{Bourgain_AP3_new}
{\sc J.~Bourgain, }
{\em Roth's Theorem on Progressions Revisited, } J. Anal. Math., {\bf 104} (2008), 155--206.


\bibitem{BourgainA+A}
{\sc J.~Bourgain, }
{\em On Aritmetic Progressions in Sums of Sets of Integers, }
A Tribute of Paul Erd\"{o}s, 105--109, Cambridge University Press, Cambridge, 1990.


\bibitem{Bukh_dilates}
{\sc B.~Bukh, }
{\em Sums of dilates, }
Combin. Probab. Comput., {\bf 17}:5 (2008), 627--639.


\bibitem{c1}
{\sc M.--C. Chang, }
{\em A polynomial bound in Freiman's theorem, } Duke Math. J. \textbf{113}:3 (2002), 399--419.


\bibitem{GT_F_2^n}
{\sc B.~Green, T.~Tao, }
{\em Freiman's theorem in finite fields via extremal set theory, }
Combinatorics, Probability and Computing {\bf 18} (2009), 335--355.


\bibitem{GT_PFRC=U_k}
{\sc B.~Green, T.~Tao, }
{\em An equavalence between inverse sumset theorems and inverse conjectures for the $U^3$--norms, }
Mathematical Proceedings of the Cambridge Philosophical Society {\bf 149}:1 (2010).


\bibitem{LY_diss}
{\sc V.F.~Lev, R.~Yuster, }
{\em On the size of dissociated bases, }  The Electronic Journal of Combinatorics 18 (1) (2011), \#P117.



\bibitem{LS}
{\sc T.~{\L}uczak, T.~Schoen, }
{\em On a problem of Konyagin, }
 vol. 134 (2008), 101--109.


\bibitem{Rudin_book}
{\sc W.~Rudin, } {\em Fourier analysis on groups,}  Wiley 1990 (reprint of the 1962 original).

\bibitem{Ruzsa-solving-I}
{\sc I.Z.~Ruzsa, } {\em Solving linear equations in sets of integers. I,} Acta Arith., {\bf 65} (1993), 259--282.

\bibitem{Ruzsa_book}
{\sc I.Z.~Ruzsa, } {\em Sumsets and structure, } book.


\bibitem{Ruzsa_card}
{\sc I.Z.~Ruzsa, } {\em Cardinality questions about sumsets, }
Additive combinatorics, ser. CRM Proc. Lecture Notes. Providence, RI: Amer. Math. Soc {\bf 43} (2007), 195--205.

\bibitem{sanders-shkredov}
{\sc T.~Sanders,} {\em On a theorem of Shkredov,} Online J. Anal. Comb. No. 5 (2010), Art. 5, 4 pp.


\bibitem{sanders-log}
{\sc T.~Sanders,} {\em Structure in sets with logarithmic doubling, } Canad. Math. Bull., to appear..


\bibitem{sanders-bogolubov}
{\sc T.~Sanders,} {\em On the Bogolubov--Ruzsa Lemma,} Anal. PDE 5 (2012), 627--655.


\bibitem{Sanders_0.75}
{\sc T.~Sanders,} {\em On certain other sets of integers, }  J. Anal. Math. 116 (2012), 53–-82.


\bibitem{Sanders_1}
{\sc T.~Sanders,} {\em On Roth's Theorem on Progressions, }  Ann. of Math. (2) 174 (2011), 619–-636.


\bibitem{schoen-freiman}
{\sc T.~Schoen,} {\em Near optimal bounds in Freiman's theorem,}
Duke Math. J. {\bf 158} (2011), 1--12.

\bibitem{schoen-shkredov-higher}
{\sc T.~Schoen  I.~D.~Shkredov, }
{\em Higher moments of convolutions,} J. of Number Theory, {\bf 133} (2013), 1693--1737.

\bibitem{Sh_doubling}
{\sc I.~D.~Shkredov, }
{\em On Sets with Small Doubling, } Mat. Zametki, {\bf 84}:6 (2008), 927--947.

\bibitem{Sh_mixed}
{\sc I.~D.~Shkredov, }
{\em Some new results on higher energies, } Transactions of MMS, {\bf 74}:1 (2013), 35--73.


\bibitem{sv}
{\sc I.~D.~Shkredov, I.~V.~V'ugin,} {\em On additive shifts of
multiplicative subgroups,} Mat. Sbornik, {\bf 203}:6 (2012), 81--100.


\bibitem{sy}
{\sc I.~D.~Shkredov, S.~Yekhanin,}
{\em Sets with large additive energy and symmetric sets,}
Journal of Combinatorial Theory, Series A {\bf 118} (2011) 1086--1093.



\bibitem{tv}
{\sc T.~Tao, V.~Vu, }{\em Additive combinatorics,} Cambridge University Press 2006.


\end{thebibliography}
\end{document}